\documentclass[a4paper,11pt,twoside,reqno]{amsart}

\usepackage[utf8]{inputenc}
\usepackage[plainpages=false,pdfpagelabels=true]{hyperref}
\usepackage{amssymb,amsthm}
\usepackage{amsmath}
\usepackage[margin=1in]{geometry}
\usepackage{dutchcal}
\usepackage{xcolor}

\newtheorem{Theorem}{Theorem}[section]
\newtheorem{Prop}[Theorem]{Proposition}

\newtheorem{Cor}[Theorem]{Corollary}
\theoremstyle{definition}

\newtheorem{Bem}[Theorem]{Remark}

\newcommand{\tr}{\operatorname{Tr}}
\newcommand{\Ric}{\operatorname{Ric}}
\newcommand{\Id}{\operatorname{Id}}
\newcommand{\Scal}{\operatorname{Scal}}
\newcommand{\Hess}{\operatorname{Hess}}

\newcommand{\Div}{\operatorname{div}}
\newcommand{\Index}{\operatorname{Index}}
\newcommand{\Nullity}{\operatorname{Nullity}}

\allowdisplaybreaks[1]

\renewcommand{\epsilon}{\varepsilon}

\numberwithin{equation}{section}

\newtheorem{innercustomgeneric}{\customgenericname}
\providecommand{\customgenericname}{}
\newcommand{\newcustomtheorem}[2]{%
  \newenvironment{#1}[1]
  {%
   \renewcommand\customgenericname{#2}%
   \renewcommand\theinnercustomgeneric{##1}%
   \innercustomgeneric
  }
  {\endinnercustomgeneric}
}

\newcustomtheorem{customthm}{Theorem}

\newcustomtheorem{customprop}{Proposition}

\title[Conformal-biharmonic stability of the identity map]{On the conformal-biharmonic stability of the identity map of Einstein manifolds}
\author{Volker Branding}
\address{University of Rostock, Institute of Mathematics\\
Ulmenstraße 69, 18057 Rostock, Germany}
\email{volker.branding@uni-rostock.de}

\author{Simona Nistor}
\address{Al.I. Cuza University of Iasi\\
Faculty of Mathematics, Blvd. Carol I, 11, 700506 Iasi, Romania}
\email{nistor.simona@ymail.com}

\author{Cezar Oniciuc}
\address{Al.I. Cuza University of Iasi\\
Faculty of Mathematics, Blvd. Carol I, 11, 700506 Iasi, Romania}
\email{oniciucc@uaic.ro}

\thanks{The first author gratefully acknowledges the support of the Austrian Science Fund (FWF) through the project ``Geometric Analysis of Biwave Maps'' ({\tt DOI: 10.55776/P34853}).}


\begin{document}

\begin{abstract}
	
The identity map of an Einstein manifold is a critical point of both the classical energy functional and the conformal-bienergy functional. In this paper, we investigate the conformal-biharmonic stability of the identity map of compact Einstein manifolds of dimension at least four and with nonnegative scalar curvature, and we compare it with the harmonic stability, when the identity map is considered as a harmonic map. Somewhat surprisingly, we show that the conformal-biharmonic index coincides with the harmonic index, with a single notable exception: the four-dimensional Euclidean sphere. In this case, the identity map is unstable with respect to the energy functional, as shown independently by Mazet and Smith, whereas it is stable with respect to the conformal-bienergy functional.

\end{abstract}

\maketitle

\section{Introduction}

Conformal-biharmonic maps were first introduced by Bérard in \cite{MR2449641, MR2722777, MR3028564} as a higher-order generalization of the classical harmonic maps between Riemannian manifolds (see, for example, \cite{MR2044031, MR164306, MR2389639}). Similarly to biharmonic maps, conformal-biharmonic maps are characterized by a fourth-order nonlinear elliptic partial differential equation. Any harmonic map is automatically biharmonic, but not always conformal-biharmonic. Another important difference between conformal-biharmonic and biharmonic maps is that the former are invariant under conformal changes of the domain metric in dimension four, whereas biharmonic maps do not enjoy such conformal invariance in any dimension. This conformal invariance property parallels that of harmonic maps, which are invariant under conformal transformations of the domain in dimension two. Harmonic maps, biharmonic maps and conformal-biharmonic maps arise as solutions to variational problems, being critical points of the classical energy functional $E$, bienergy functional $E_2$ and conformal-bienergy functional $E_2^c$, respectively. All these functionals are defined on the space of smooth maps between two given Riemannian manifolds.

Conformal-biharmonic maps have mostly been studied from an analytical point of view (see, for example, \cite{MR2996776, MR4142862, branding2025, MR2124627, MR4847311, MR2094320}).
However, very recently, these maps have been investigated from the perspective of submanifold theory and, in addition, an extensive analysis of their stability was conducted (see, for example, \cite{BMNOR, BNO}). In particular, the following result concerning the conformal-biharmonic stability of the identity map of the Euclidean sphere has been obtained.

\begin{Theorem}[\cite{BNO}]\label{stability-identity}
	We consider $\Id:\mathbb{S}^m\to\mathbb{S}^{m}$ to be the identity map of the Euclidean unit sphere $\mathbb{S}^m$, thought of as a conformal-biharmonic map. Then, we have
	\begin{itemize}
		\item[i)] if $m=1$ or $m=3$, the conformal-biharmonic index and the conformal-biharmonic nullity of $\mathbb{S}^m$ are $0$ and $m(m+1)/2$, respectively;
		\item[ii)] if $m=2$ or $m=4$, the conformal-biharmonic index and the conformal-biharmonic nullity of $\mathbb{S}^m$ are $0$ and $(m+1)(m+2)/2$, respectively;
		\item[iii)] if $m\geq 5$, the conformal-biharmonic index and the conformal-biharmonic nullity of $\mathbb{S}^m$ are $m+1$ and $m(m+1)/2$, respectively.
	\end{itemize}
\end{Theorem}

The identity map of the Euclidean sphere, as any harmonic map, is automatically stable as a biharmonic map. However, in light of the classical results of Mazet and Smith on its stability as a harmonic map, 
Theorem \ref{stability-identity} appears particularly interesting.  

\begin{Theorem}[\cite{MR336767, MR0375386}]\label{stability-identity-harmonic}
	We consider $\Id:\mathbb{S}^m\to\mathbb{S}^{m}$ to be the identity map of $\mathbb{S}^m$, thought of as a harmonic map. Then, we have
	\begin{itemize}
		\item[i)] if $m=1$, the harmonic index and the harmonic nullity of $\mathbb{S}^m$ are $0$ and $1$, respectively;
		\item[ii)] if $m=2$, the harmonic index and the harmonic nullity of $\mathbb{S}^m$ are $0$ and $6$, respectively;
		\item[iii)] if $m\geq 3$, the harmonic index and the harmonic nullity of $\mathbb{S}^m$ are $m+1$ and $m(m+1)/2$, respectively.
	\end{itemize}
\end{Theorem}

Therefore, for $m\leq 4$, the identity map is ``more stable'' with respect to the conformal-bienergy functional than with respect to the classical energy functional.

In this note, we continue the study of the stability of conformal-biharmonic maps, focusing on the identity map of a compact Einstein manifold with nonnegative scalar curvature and dimension at least four. We first determine the expression of the Jacobi operator of the conformal-bienergy functional associated with an arbitrary conformal-biharmonic map. Then, we focus on the case of the identity map of Einstein manifolds. Since this map is also a critical point of the classical energy functional, it is natural to compare its index and stability with respect to the two functionals. Somewhat surprisingly, the index and nullity of the identity map coincide for both functionals in almost all cases, with a single notable exception -- the one mentioned above. More precisely, we obtain

\begin{customthm}{\ref{main-theorem}}
	Let $\left(M^m,g\right)$ be a compact Einstein manifold of positive scalar curvature with $m\geq 4$ and consider $\Id:\left(M,g\right)\to \left(M,g\right)$ to be the identity map. Then,
	$$
	\Index_{E_2^c}\left(\Id\right)=\Index_{E}\left(\Id\right)\qquad\text{and}\qquad
	\Nullity_{E_2^c}\left(\Id\right)=\Nullity_{E_2} \left(\Id\right),
	$$
	with only one exception, namely up to a rescaling of the round metric, $\left(M,g\right)=\mathbb{S}^4$. In this case,
	$$
	\Index_{E_2^c}\left(\Id\right)=0<\Index_{E}\left(\Id\right)=5\qquad\text{and}\qquad \Nullity_{E_2^c}\left(\Id\right)=15>\Nullity_{E}\left(\Id\right)=10.
	$$
\end{customthm}

Thus, in the case of the identity map of Einstein manifolds, the conformal-bienergy functional enjoys ``stronger stability properties'' than the classical energy functional.

We also obtain the following interesting result which, in combination with Proposition \ref{prop-non-negativeE2c}, says that in the homotopy class of the identity map of the Euclidean sphere $\mathbb{S}^m$, $m\geq 5$, the infimum of $E_2^c$ is zero.

\begin{customthm}{\ref{prop-e2c-small}}
For any $\varepsilon>0$, there exists a map $\phi$ homotopic to the identity map $\Id:\mathbb{S}^m\to\mathbb{S}^m$, $m\geq 5$, with $E_2^c(\phi)<\varepsilon$.
\end{customthm}

\textbf{Conventions.} Throughout this article we shall use the following sign conventions: 
for the Riemannian curvature tensor field we use 
$$
R(X,Y)Z=\left[\nabla_X,\nabla_Y\right]Z-\nabla_{[X,Y]}Z \qquad \text{and} \qquad R(X,Y,Z,V)=g(R(X,Y)V,Z),
$$ 
where $X$, $Y$, $Z$, $V$ are vector fields. For the Ricci tensor field we employ
$$
g(\Ric(X),Y)=\Ric(X,Y)=\tr \left\{Z\to R(Z,X)Y\right\},
$$
and the scalar curvature is given by
$$
\Scal=\tr\Ric.
$$
The trace is taken with respect to the domain metric and we write $\tr$ instead of $\tr_g$. Moreover, we also use the following convention for the covariant derivative of the curvature tensor field
\begin{align*}
\left(\nabla R\right)(X,Y,Z,V)&=\left(\nabla_X R\right)(Y,Z,V)\\
&=\nabla_X R(Y,Z)V-R\left(\nabla_X Y,Z\right)V-R\left(Y,\nabla_X Z\right)V-R(Y,Z)\nabla_X V.
\end{align*}
To avoid any ambiguity, as we work with various pull-back bundles, we use the notation $\nabla^\phi$ for the connection on the pull-back bundle $\phi^{-1}TN$.

For the rough Laplacian on the pull-back bundle $\phi^{-1} TN$ we employ the geometers sign convention
$$
\Delta^\phi=-\tr\left(\left(\nabla^\phi\right)^2-\nabla^{\phi}_\nabla\right)=-\tr\left(\nabla^\phi\nabla^\phi-\nabla^\phi_\nabla\right).
$$
In particular, this implies that the Laplace operator has a nonnegative spectrum. For the metric $g$ on $TM$ and also for the metric on the other vector bundles involved, we use the same symbol $\langle\cdot,\cdot\rangle$. 

In order to perform various computations, we fix an arbitrary point $p\in M$ and denote by $\left\{X_i\right\}_{i\in\overline{1,m}}$ a geodesic frame field around it. We work locally and, at the end, we evaluate the expressions involved at $p$. 

All manifolds are assumed to be smooth, connected and without boundary. Thus, in our paper, all compact manifolds are closed.

\section{Preliminaries}

We recall that the classical \emph{energy functional} $E$ is defined on the set of all smooth maps between two fixed Riemannian manifolds $\phi:\left(M^m,g\right)\to\left(N^n,h\right)$ and is given by
\begin{align*}
E(\phi):=\frac{1}{2}\int_M|d\phi|^2 \ v_g,
\end{align*}
where we assume for simplicity that $M$ is compact. The critical points of this functional are called \emph{harmonic maps}, and they are characterized by the vanishing of the tension field
\begin{align*}
\tau(\phi):=\tr\nabla^\phi d\phi, \qquad \tau(\phi)\in C\left(\phi^{-1}TN\right).
\end{align*}
The harmonic map equation, i.e., $\tau(\phi)=0$, is a second-order nonlinear elliptic partial differential equation.

A natural higher-order generalization of harmonic maps was introduced in the 1980s by Jiang in \cite{MR886529} (see also \cite{MR2640582}), where the author defined the \emph{bienergy functional}
\begin{align*}
E_2(\phi):=\frac{1}{2}\int_M|\tau(\phi)|^2 \ v_g,
\end{align*}
and called its critical points \emph{biharmonic maps}. The Euler-Lagrange equation for this functional, i.e., the biharmonic map equation, is given by the vanishing of the bitension field
\begin{align*}
\tau_2(\phi):=-\Delta^\phi\tau(\phi)-\tr R^N(d\phi(\cdot), \tau(\phi))d\phi(\cdot), \qquad \tau_2(\phi)\in C\left(\phi^{-1}TN\right).
\end{align*}
The biharmonic map equation, i.e., $\tau_2(\phi)=0$, is a fourth-order nonlinear elliptic partial differential equation, and it is evident that any harmonic map is biharmonic.

While the harmonic map equation is invariant under conformal transformations of the domain in dimension two, the biharmonic map equation does not enjoy this property in any dimension. Thus, it is natural to search for higher-order energy functionals that are conformally invariant in even dimensions, the most convenient setting being dimension four. In 2008, Bérard introduced a new family of functionals $E_r^c$ which, in dimension $m=2r$, are invariant under conformal changes of the domain metric (see \cite{MR2449641,MR2722777}). For $r=1$, Bérard’s functional coincides with the classical energy functional $E$. For $r=2$, the functional $E_2^c$ is called \emph{conformal-bienergy functional}, or simply, \emph{c-bienergy functional}, and it is given by 
\begin{align}\label{c-bienergy}
	E_2^c(\phi):=&\frac{1}{2}\int_{M}\left|\tau(\phi)\right|^2+\frac{2}{3}\Scal |d\phi|^2-2\tr \langle d\phi(\Ric(\cdot)),d\phi(\cdot)\rangle\ v_g \nonumber\\
	=&E_2(\phi)+\int_{M}\frac{1}{3}\Scal |d\phi|^2-\tr \langle d\phi(\Ric(\cdot)),d\phi(\cdot)\rangle\ v_g.
\end{align}
The critical points of the c-bienergy functional are characterized by the vanishing of the c-bitension field
$$
\tau_2^c(\phi):=\tau_2 (\phi)-\frac{2}{3}\Scal \tau(\phi)+2\tr\left(\nabla^\phi d\phi\right)\left(\Ric(\cdot),\cdot\right)+\frac{1}{3}d\phi(\nabla \Scal), \qquad \tau_2^c(\phi)\in C\left(\phi^{-1}TN\right),
$$
and they are called \emph{conformal-biharmonic maps}, or simply, \emph{c-biharmonic maps}. The c-biharmonic map equation, i.e., $\tau_2^c(\phi)=0$, also is a fourth-order nonlinear elliptic partial differential equation. A quick inspection of the c-biharmonic map equation shows that not every harmonic map is c-biharmonic. This fact is illustrated by the identity map of a Riemannian manifold which is c-biharmonic if and only if the manifold has constant scalar curvature.

\section{The second variation formula of the c-bienergy functional}


Let $\phi:\left(M^m,g\right)\to\left(N^n,h\right)$ be a c-biharmonic map and consider $\Phi$ to be a two-parameter smooth variation of $\phi$, i.e., a smooth map 
$$
\Phi:\mathbb{R}\times\mathbb{R}\times M\to N, \qquad \Phi(s,t,p)=\phi_{s,t}(p),\qquad (s,t,p)\in \mathbb{R}\times\mathbb{R}\times M,
$$
such that $\Phi(0,0,p)=\phi_{0,0}(p)=\phi(p)$. The variational vector fields associated with the variation $\left\{\phi_{s,t}\right\}_{(s,t)\in\mathbb{R}\times\mathbb{R}}$ of $\phi$ are 
$$
V(p)= \left.\frac{d}{ds}\right|_{s=0}\left\{\phi_{s,0}(p)\right\}=\left(\left.\frac{d}{ds}\right|_{s=0}\left\{\phi_{s,0}^\alpha(p)\right\}\right)\left(\partial y^\alpha\right)_{\phi(p)}=d\Phi_{(0,0,p)}\left(\partial s\right)\in T_{\phi(p)}N
$$ 
and
$$
W(p)= \left.\frac{d}{dt}\right|_{t=0}\left\{\phi_{0,t}(p)\right\}=\left(\left.\frac{d}{dt}\right|_{t=0}\left\{\phi_{0,t}^\alpha(p)\right\}\right)\left(\partial y^\alpha\right)_{\phi(p)}=d\Phi_{(0,0,p)}\left(\partial t\right)\in T_{\phi(p)}N,
$$ 
where $\left(\phi^\alpha_{s,t}\right)_{\alpha=\overline{1,n}}$ are the components of $\phi_{s,t}$ in the local coordinates $\left(y^\alpha\right)$ on the target manifold.

It is well known (see \cite{MR336767, MR0375386} and also \cite{MR2044031, MR1252178}) that the Jacobi operator $J$ of the energy functional $E$ associated with a harmonic map $\phi$ is given by
\begin{equation}\label{jacobi-energy}
	J(W)=\Delta^\phi W+\tr R^N(d\phi(\cdot),W)d\phi(\cdot),
\end{equation}
where $W$ is an arbitrary section in the pullback-bundle $\phi^{-1}TN$, i.e., $W\in C\left(\phi^{-1}TN\right)$. It is a second-order linear elliptic $L^2$-selfadjoint operator and the Hessian of $\phi$, defined as 
$$
\Hess_E\left(\phi\right)(V,W)=
\left.\frac{\partial^2}{\partial s\partial t}\right|_{(s,t)=(0,0)}\left\{E\left(\phi_{s,t}\right) \right\}=\int_M \langle J(W),V\rangle\ v_g,
$$
is a symmetric bilinear form on $C\left(\phi^{-1}TN\right)$.

We emphasize that, even when $\phi$ is not harmonic, we shall still denote by $J$ the differential operator defined by \eqref{jacobi-energy}.

It is also known (see \cite{MR886529} and the translation \cite{MR2640582}) that the Jacobi operator $J_2$ of the bienergy functional $E_2$ associated with a biharmonic map $\phi$ is given by
\begin{align}\label{jacobi-bienergy}
 J_2(W)&=J^2(W)-R^N(\tau(\phi), W)\tau(\phi)-\tr \left(\nabla^N R^N\right)\left(\tau(\phi),d\phi(\cdot),W, d\phi(\cdot)\right) \nonumber \\
 &\quad+\tr \left\{\vphantom{R^N\left(W, d\phi(\cdot)\right)\nabla^\phi_\cdot \tau(\phi)}\left(\nabla^N R^N\right)\left(d\phi(\cdot),\tau(\phi),d\phi(\cdot),W\right)\right. \nonumber\\
 &\qquad\qquad\left.+2R^N\left(W, d\phi(\cdot)\right)\nabla^\phi_{\left(\cdot\right)} \tau(\phi)+2R^N\left(\tau(\phi),d\phi(\cdot)\right)\nabla^\phi_{\left(\cdot\right)} W\right\},
\end{align}
where $J^2(W)=J(J(W))$. 

It is easy to check that each of the first three terms on the right-hand side of \eqref{jacobi-bienergy} is $L^2$-selfadjoint. In order to prove the selfadjointness of the sum of the last three terms with respect to the $L^2$-inner product, by standard computations using the first Bianchi identity, we notice that
\begin{align*}
&\langle\tr \vphantom{R^N\left(W, d\phi(\cdot)\right)\nabla^\phi_\cdot \tau(\phi)}\left(\nabla^N R^N\right)\left(d\phi(\cdot),\tau(\phi),d\phi(\cdot),W\right),V\rangle = \\
&= \Div Y + \sum_{i=1} ^m \left\{ \langle R^N\left( d\phi(X_i), \nabla^\phi_{X_i}\tau(\phi)\right)W, V \rangle+\langle R^N\left(d\phi (X_i), \tau(\phi)\right)W, \nabla^\phi_{X_i}V \rangle \right. \\
&\qquad\qquad\qquad\quad-\left.\langle R^N(\tau(\phi),d\phi(X_i))\nabla^\phi_{X_i}W, V\rangle \right\} \\
&= \Div Y + \sum_{i=1} ^m \left\{ \langle R^N \left(W, \nabla^\phi_{X_i}\tau(\phi)\right)d\phi(X_i),V\rangle-\langle R^N(W,d\phi(X_i))\nabla^\phi_{X_i}\tau(\phi),V\rangle\right. \\
&\qquad\qquad\qquad\quad+\left.\langle R^N(d\phi (X_i), \tau(\phi))W, \nabla^\phi_{X_i}V \rangle-\langle R^N(\tau(\phi),d\phi(X_i))\nabla^\phi_{X_i}W, V\rangle \right\},
\end{align*}
where 
$$
Y=-\tr\left\langle R^N(d\phi(\cdot),\tau(\phi))W,V\right\rangle (\cdot).
$$
Thus, 
\begin{align*}
&\langle\tr \left\{\vphantom{R^N\left(W, d\phi(\cdot)\right)\nabla^\phi_\cdot \tau(\phi)}\left(\nabla^N R^N\right)\left(d\phi(\cdot),\tau(\phi),d\phi(\cdot),W\right)+2R^N\left(W, d\phi(\cdot)\right)\nabla^\phi_{(\cdot)} \tau(\phi)+2R^N\left(\tau(\phi),d\phi(\cdot)\right)\nabla^\phi_{(\cdot)} W\right\},V\rangle =\\
&= \Div Y + \sum_{i=1} ^m \left\{ R^N \left(W, \nabla^\phi_{X_i}\tau(\phi),V,d\phi(X_i)\right)
+R^N\left(V,\nabla^\phi_{X_i}\tau(\phi),W,d\phi(X_i)\right)\right. \\
&\qquad\qquad\qquad\quad+\left. R^N\left(d\phi (X_i), \tau(\phi), \nabla^\phi_{X_i}V,W\right)+R^N\left(\tau(\phi),d\phi(X_i),V,\nabla^\phi_{X_i}W\right) \right\}.
\end{align*}
Further, using Stokes' Theorem, we can state the following.

\begin{Theorem}
Let $\phi:\left(M^m,g\right)\to\left(N^n,h\right)$ be a biharmonic map and consider a two-parameter smooth variation of $\phi$. Then, 
\begin{align*}\label{second-variation-bienergy-1}
\Hess_{E_2}\left(\phi\right)(V,W)&=
	\left.\frac{\partial^2}{\partial s\partial t}\right|_{(s,t)=(0,0)}\left\{E_2\left(\phi_{s,t}\right) \right\} \nonumber \\
	&=\int_{M}\langle J_2(W),V\rangle\ v_g \nonumber\\
	&=\int_{M}\langle J^2(W)-R^N(\tau(\phi), W)\tau(\phi) -\tr \left(\nabla^N R^N\right)\left(\tau(\phi),d\phi(\cdot),W, d\phi(\cdot)\right), V\rangle \nonumber \\
    & \qquad+ \tr\left\{R^N\left(W,\nabla^\phi_{(\cdot)} \tau(\phi),V, d\phi(\cdot)\right)+R^N\left(V,\nabla^\phi_{(\cdot)} \tau(\phi),W, d\phi(\cdot)\right)\right\} \nonumber \\
	&\qquad+ \tr\left\{R^N\left(d\phi\left(\cdot\right),\tau(\phi),\nabla^\phi_{(\cdot)} V, W\right)+R^N\left(d\phi\left(\cdot\right),\tau(\phi),\nabla^\phi_{(\cdot)} W, V\right)\right\} \ v_g
\end{align*}
and therefore $J_2$ is a fourth-order linear elliptic $L^2$-selfadjoint operator with the leading term $\left(\Delta^\phi\right)^2$.
\end{Theorem}

We again note that, even when $\phi$ is not biharmonic, we shall still denote by $J_2$ the differential operator defined by \eqref{jacobi-bienergy}.

When the map $\phi$ is not biharmonic, the second variation formula of the bienergy functional reads
\begin{equation}\label{second-variation-bienergy}
\left.\frac{\partial^2}{\partial s\partial t}\right|_{(s,t)=(0,0)}\left\{E_2\left(\phi_{s,t}\right) \right\}=\int_{M}\langle J_2(W),V\rangle+\langle \left.\nabla^{\Phi}_{\partial t}V_t\right|_{(s,t)=(0,0)}, \tau_2(\phi)\rangle\ v_g, 
\end{equation}
where
$$
V_t(p)=\left.\frac{d}{ds}\right|_{s=0}\left\{\phi_{s,t}(p)\right\}=d\Phi_{(0,t,p)}\left(\partial s\right).
$$
In order to compute the Jacobi operator $J_2^c$ of the c-bienergy functional $E_2^c$ associated with a c-biharmonic map $\phi$, it suffices to determine the contribution to the second variation formula of the two additional terms in $E_2^c$.

A straightforward computation shows that
\begin{align*}
	\left.\frac{\partial}{\partial s}\right|_{s=0} \left\{\int_{M} \Scal \left|d\phi_{s,t}\right|^2\ v_g \right\} &=\int_M \left(\partial s\right)_{(0,t)} \left\{\Scal\sum_{i=1}^m\langle d\Phi\left(X_i\right), d\Phi\left(X_i\right)\rangle\right\}\ v_g \nonumber\\
	&=2\int_{M}\Scal \sum_{i=1}^{m}\left.\langle\nabla^{\Phi}_{\partial s} d\Phi\left(X_i\right), d\Phi\left(X_i\right)\rangle\right|_{s=0}\ v_g \nonumber\\
	&=2\int_{M}\Scal \sum_{i=1}^{m}\left.\langle\nabla^{\Phi}_{X_i} d\Phi\left(\partial s\right)+d\Phi\left(\left[\partial s, X_i\right]\right), d\Phi\left(X_i\right)\rangle\right|_{s=0}\ v_g \nonumber\\
	&=2\int_{M}\Scal \sum_{i=1}^{m}\langle \nabla^{\phi_{0,t}}_{X_i} V_t,d\phi_{0,t}\left(X_i\right)\rangle\ v_g\nonumber\\
	&=2\int_{M} \Scal \sum_{i=1}^{m}\left\{X_i\langle V_t,d\phi_{0,t}\left(X_i\right)\rangle -\langle V_t, \nabla^{\phi_{0,t}}_{X_i} d\phi_{0,t}\left(X_i\right)\rangle\right\}\ v_g.
 \end{align*}
Thus, we get
\begin{equation}\label{second-variation-2}
\left.\frac{\partial}{\partial s}\right|_{s=0} \left\{\int_{M} \Scal \left|d\phi_{s,t}\right|^2\ v_g \right\} =-2\int_{M} \Scal\left(\langle V_t, \tau(\phi_{0,t})\rangle-\sum_{i=1}^{m}X_i\langle V_t,d\phi_{0,t}\left(X_i\right)\rangle\right)\ v_g.
\end{equation}
We next introduce the vector field 
$$
Y_{1,t}=\Scal\tr\langle V_t, d\phi_{0,t}(\cdot)\rangle(\cdot).
$$
A direct computation yields 
\begin{equation}\label{div-1}
\Div Y_{1,t}=\langle V_t,d\phi_{0,t}(\nabla\Scal)\rangle+\Scal\sum_{i=1}^mX_i\langle V_t,d\phi_{0,t}\left(X_i\right)\rangle.
\end{equation}
Replacing \eqref{div-1} in \eqref{second-variation-2}, we get
\begin{equation*}
	\left.\frac{\partial}{\partial s}\right|_{s=0} \left\{\int_{M} \Scal \left|d\phi_{s,t}\right|^2\ v_g\right\}=-2\int_{M}\langle V_t,d\phi_{0,t}(\nabla\Scal)+\Scal\tau(\phi_{0,t}) \rangle \ v_g.
\end{equation*}
We then differentiate the above identity with respect to $t$ and evaluate at $t=0$, which yields 
\begin{align}\label{contribution-of-first-additional-term}
	&\left.\frac{\partial^2}{\partial s\partial t}\right|_{(s,t)=(0,0)}\left\{\int_{M} \Scal \left|d\phi_{s,t}\right|^2\ v_g \right\}=\nonumber\\
	&=-2\int_{M}\langle \left.\nabla_{\partial t}^{\Phi} V_t\right|_{(s,t)=(0,0)}, d\phi(\nabla \Scal)+\Scal \tau(\phi)\rangle \nonumber\\
	&\qquad\qquad+\langle V,\left.\nabla^\Phi_{\partial t}\left(d\Phi\left(\nabla\Scal\right)+\Scal\tau(\phi_{0,t})\right)\right|_{(s,t)=(0,0)}\rangle \ v_g \nonumber\\
	&=-2\int_{M}\langle \left.\nabla_{\partial t}^{\Phi} V_t\right|_{(s,t)=(0,0)}, d\phi(\nabla \Scal)+\Scal \tau(\phi)\rangle +\langle V,\nabla^\phi_{\nabla \Scal}W-\Scal J(W)\rangle \ v_g.
\end{align}
We now turn to the contribution of the second additional term in $E_2^c$ to the second variation formula. For the geodesic frame field $\left\{X_i\right\}_{i\in\overline{1,m}}$ around an arbitrarily fixed point $p\in M$ we have $\nabla_{\Ric\left(X_i\right)}X_i=0$ at $p$ and thus, a standard computation leads to
\begin{align}\label{second-additional-term-1}
\left.\frac{\partial}{\partial s}\right|_{s=0}&\left\{\int_{M} \tr \langle d\phi_{s,t}(\Ric(\cdot)), d\phi_{s,t}(\cdot)\rangle \ v_g\right\}=\nonumber\\
&=\int_M \left(\partial s\right)_{(0,t)} \left\{\sum_{i=1}^m\langle d\Phi\left(\Ric\left(X_i\right)\right), d\Phi\left(X_i\right)\rangle\right\}\ v_g \nonumber\\
&=\int_M \sum_{i=1}^m \left\{\Ric\left(X_i\right)\langle V_t, d\phi_{0,t}\left(X_i\right)\rangle - 2 \langle V_t, \left(\nabla^{\phi_{0,t}}d\phi_{0,t}\right)\left(\Ric\left(X_i\right),X_i\right)\rangle\right.\nonumber \\
&\quad\qquad\qquad\left. +X_i\langle V_t, d\phi_{0,t}\left(\Ric\left(X_i\right)\right)\rangle-\langle V_t, d\phi_{0,t}\left(\left(\nabla\Ric\right)\left(X_i,X_i\right)\right)\rangle\vphantom{\left(\nabla^{\phi_{0,t}}d\phi_{0,t}\right)}\right\}\ v_g.
\end{align}
We define the vector fields 
$$
Y_{2,t}=\tr \langle V_t, d\phi_{0,t}(\cdot)\rangle \Ric(\cdot)
$$
and 
$$
Y_{3,t}=\tr\langle V_t, d\phi_{0,t}\left(\Ric(\cdot)\right)\rangle(\cdot).
$$
Using that 
\begin{align*}
	\sum_{i,j=1}^m \langle \left(\nabla\Ric\right)\left(X_i,X_j\right),X_i\rangle X_j&= \sum_{i,j=1}^m \langle \nabla_{X_i}\Ric\left(X_j\right),X_i\rangle X_j\\
	&=\sum_{i,j=1}^m X_i\langle\Ric\left(X_j\right), X_i\rangle X_j \\
	&=\sum_{i,j=1}^m X_i\langle\Ric\left(X_i\right), X_j\rangle X_j \\
	&=\sum_{i,j=1}^m\langle\left(\nabla\Ric\right)\left(X_i,X_i\right), X_j\rangle X_j \\
	&=\sum_{i=1}^m \left(\nabla\Ric\right)\left(X_i,X_i\right)
\end{align*}
at the point $p$, we derive
$$
\Div Y_{2,t}=\sum_{i=1}^m\left\{\Ric\left(X_i\right)\langle V_t, d\phi_{0,t}\left(X_i\right)\rangle+\langle V_t, d\phi_{0,t}\left(\left(\nabla\Ric\right)\left(X_i,X_i\right)\right)\rangle\right\}
$$
and
$$
\Div Y_{3,t}=\sum_{i=1}^m X_i\langle V_t, d\phi_{0,t}\left(\Ric\left(X_i\right)\right)\rangle.
$$
Replacing the above two relations in \eqref{second-additional-term-1} one obtains
\begin{align*}
	\left.\frac{\partial}{\partial s}\right|_{s=0} &\left\{\int_{M} \tr \langle d\phi_{s,t}(\Ric(\cdot)), d\phi_{s,t}(\cdot)\rangle \ v_g\right\}=\\
	&=-2\int_M \sum_{i=1}^m \langle V_t,\left(\nabla^{\phi_{0,t}}d\phi_{0,t}\right)\left(\Ric\left(X_i\right),X_i\right)+d\phi_{0,t}\left(\left(\nabla\Ric\right)\left(X_i,X_i\right)\right)\rangle\ v_g.
\end{align*}
Differentiating the above identity once more with respect to $t$, evaluating at $t=0$, and using the formula
$$
\Div\Ric=\frac{1}{2}\nabla\Scal,
$$
after some computations, one achieves
\begin{align}\label{contribution-of-the-second-additional-term}
	\left.\frac{\partial^2}{\partial s\partial t}\right|_{(s,t)=(0,0)}&\left\{\int_{M} \tr \langle d\phi_{s,t}\left(\Ric(\cdot)\right),d\phi_{s,t}(\cdot)\rangle\ v_g \right\}=\nonumber\\
	&=-2\int_{M} \langle \left.\nabla_{\partial t}^{\Phi} V_t\right|_{(s,t)=(0,0)}, \tr\left(\nabla^\phi d\phi\right)\left(\Ric(\cdot),\cdot\right)+\frac{1}{2}d\phi(\nabla \Scal)\rangle\nonumber\\
	&\qquad\qquad+\langle V,\tr \left(\left(\nabla^\phi\right)^2 W\right)\left(\Ric(\cdot),\cdot\right)-\tr R^N\left(d\phi\left(\Ric(\cdot)\right),W\right),d\phi(\cdot) \rangle\nonumber\\
	&\qquad\qquad +\frac{1}{2}\langle V,\nabla^\phi_{\nabla \Scal} W\rangle 
	\ v_g.
\end{align}
Using \eqref{c-bienergy}, \eqref{second-variation-bienergy}, \eqref{contribution-of-first-additional-term} and \eqref{contribution-of-the-second-additional-term}, we conclude that the second variation formula of the c-bienergy functional \eqref{c-bienergy} is given by
\begin{equation}\label{second-variation-c-bienergy}
	\left.\frac{\partial^2}{\partial s\partial t}\right|_{(s,t)=(0,0)}\left\{E_2^c\left(\phi_{s,t}\right) \right\}=\int_{M}\langle J_2^c(W),V\rangle+\langle \left.\nabla^{\Phi}_{\partial t}V_t\right|_{(s,t)=(0,0)}, \tau_2^c(\phi)\rangle\ v_g, 
\end{equation}
where $J_2^c$ represents the Jacobi operator of the c-bienergy functional and it is made up by the terms computed above. More precisely, we conclude with the following result.

\begin{Theorem}
Let $\phi:\left(M^m,g\right)\to\left(N^n,h\right)$ be a c-biharmonic map with $M$ compact. Then, the Jacobi operator of $E_2^c$ is given by 
\begin{align}\label{jacobi-c-bienergy}
	J_2^c(W)&=J_2(W)-2\tr R^N\left(d\phi\left(\Ric(\cdot)\right),W\right)d\phi(\cdot)  \nonumber\\
	&\quad +\frac{2}{3}\Scal J(W)+2\tr\left(\left(\nabla^\phi\right)^2 W\right)\left(\Ric(\cdot),\cdot\right)+\frac{1}{3}\nabla^\phi_{\nabla\Scal}W,
\end{align}
where $J$ and $J_2$ are formally the Jacobi operators of $E$ and $E_2$, respectively.
\end{Theorem}

The first two terms of the right-hand side of \eqref{jacobi-c-bienergy} are each $L^2$-selfadjoint. In order to prove the selfadjointness of the sum of the last three terms with respect to the $L^2$-inner product, we first recall that
$$
\langle J(W),V\rangle = \langle \nabla^\phi W, \nabla^\phi V\rangle+\langle \tr R^N\left(d\phi(\cdot),W\right)d\phi(\cdot),V\rangle-\Div Z, 
$$
where
$$
Z=\tr \langle \nabla^\phi_{\left (\cdot\right )} W, V \rangle \left (\cdot\right ).
$$
Then, by standard computations, we obtain
\begin{align*}
   \frac{2}{3} \langle \Scal J(W), V \rangle &= \frac{2}{3} \Scal \langle \nabla^\phi W, \nabla^\phi V\rangle+\frac{2}{3}\Scal \langle \tr R^N\left(d\phi(\cdot),W\right)d\phi(\cdot),V\rangle\\
   &\quad+\frac{2}{3}\langle \nabla^\phi_{\nabla \Scal} W,V\rangle-\frac{2}{3}\Div(\Scal Z).   
\end{align*}
For the fourth term of the right-hand side of \eqref{jacobi-c-bienergy}, we first have
\begin{align} \label{jacobi-c-bienergy-4-1}
   \langle 2 \tr \left ( \left ( \nabla ^\phi \right )^2 W \right ) (\Ric (\cdot), \cdot), V \rangle &= 2 \sum_{i=1}^m \langle \nabla^\phi_{\Ric (X_i)} \nabla^\phi_{X_i} W - \nabla^\phi_{\nabla_{\Ric (X_i)} X_i} W, V \rangle \nonumber\\
   &= 2 \sum_{i=1}^m \langle \nabla^\phi_{\Ric (X_i)} \nabla^\phi_{X_i} W, V \rangle \nonumber\\
   &= 2 \sum_{i=1}^m \left \{ \Ric (X_i) \langle \nabla^\phi_{X_i} W, V \rangle - \langle \nabla^\phi _{X_i} W, \nabla^\phi_{\Ric (X_i)} V \rangle \right \}.
\end{align}
Next, a direct computation shows that
\begin{align*}
  \Div(\Ric (Z))  &= \sum_{i=1}^m \Ric \left ( X_i \right ) \langle \nabla^\phi_{X_i} W, V \rangle + \sum_{i,k=1}^m \langle \nabla^\phi_{X_k} W, V \rangle \langle \left ( \nabla \Ric \right ) \left ( X_i, X_i \right ), X_k \rangle \nonumber\\
  &= \sum_{i=1}^m \Ric \left ( X_i \right ) \langle \nabla^\phi_{X_i} W, V \rangle + \frac 1 2 \sum_{i=1}^m \langle \nabla^\phi _{X_i} W, V \rangle \langle \nabla \Scal, X_i \rangle. \nonumber 
\end{align*}
Thus,
\begin{equation} \label{jacobi-c-bienergy-4-2}
  \sum_{i=1}^m \Ric \left ( X_i \right ) \langle \nabla^\phi_{X_i} W, V \rangle = \Div(\Ric (Z)) - \frac{1}{2} \langle \nabla^\phi _{\nabla \Scal} W, V \rangle.
\end{equation}
Replacing \eqref{jacobi-c-bienergy-4-2} in \eqref{jacobi-c-bienergy-4-1}, one gets
\begin{align*}
	\langle 2 \tr \left ( \left ( \nabla ^\phi \right )^2 W \right ) (\Ric (\cdot), \cdot), V \rangle &= 2 \Div(\Ric (Z)) - \langle \nabla^\phi _{\nabla \Scal} W, V \rangle-2\tr\langle \nabla^\phi_{(\cdot)} W,\nabla^\phi_{\Ric(\cdot)}V\rangle.
\end{align*}
Thus,
\begin{align*}
&\langle\frac{2}{3}\Scal J(W)+2\tr\left(\left(\nabla^\phi\right)^2 W\right)\left(\Ric(\cdot),\cdot\right)+\frac{1}{3}\nabla^\phi_{\nabla\Scal}W, V\rangle =  \\
&=\frac{2}{3} \Scal \langle \nabla^\phi W, \nabla^\phi V\rangle+ \frac{2}{3}\Scal\langle\tr R^N\left(d\phi\left(\Ric(\cdot)\right),W\right)d\phi(\cdot), V\rangle -\frac{2}{3}\Div(\Scal Z)+2 \Div(\Ric (Z))\\
&\quad -2\tr\langle \nabla^\phi_{(\cdot)} W,\nabla^\phi_{\Ric(\cdot)}V\rangle.
\end{align*}
Further, using Stokes' Theorem, we can state the following.
\begin{Theorem}
Let $\phi:\left(M^m,g\right)\to\left(N^n,h\right)$ be a c-biharmonic map and consider a two-parameter smooth variation of $\phi$. Then, 
\begin{align*}\label{second-variation-bienergy-1}
\Hess_{E_2^c}\left(\phi\right)(V,W)&=
	\left.\frac{\partial^2}{\partial s\partial t}\right|_{(s,t)=(0,0)}\left\{E_2^c\left(\phi_{s,t}\right) \right\} \nonumber \\
	&=\int_{M}\langle J_2^c(W),V\rangle\ v_g \nonumber\\
	&=\int_{M}\langle J_2(W)-2\tr R^N\left(d\phi\left(\Ric(\cdot)\right),W\right)d\phi(\cdot), V\rangle \nonumber \\
    & \qquad + \frac{2}{3} \Scal \langle\tr R^N\left(d\phi\left(\Ric(\cdot)\right),W\right)d\phi(\cdot), V\rangle \nonumber \\
    & \qquad + \frac{2}{3} \Scal \langle\nabla^\phi W, \nabla^\phi V\rangle - 2 \tr \langle \nabla ^\phi _{\left (\cdot \right )} W, \nabla ^\phi _{\Ric (\cdot)} V \rangle \ v_g
\end{align*}
and therefore $J_2^c$ is a fourth-order linear elliptic $L^2$-selfadjoint operator with the leading term $\left(\Delta^\phi\right)^2$.
\end{Theorem}

From the theory of linear elliptic operators on compact manifolds, it is well known that the eigenspaces of $J_2^c$ are finite dimensional and consist of smooth sections. The Hilbert space of $L^2$-sections in $\phi^{-1}TN$ is the closure of the (infinite and $L^2$-orthogonal) sum of the eigenspaces of $J_2^c$, and the spectrum of $J_2^c$ consists of a discrete sequence of real numbers, bounded below by the first eigenvalue. Thus, the index of $\Hess_{E_2^c}\left(\phi\right)$, or the \emph{conformal-biharmonic index} of $\phi$, is finite, where the conformal-biharmonic index of $\phi$ is the dimension of the largest subspace of $C\left(\phi^{-1}TN\right)$ on which $\Hess_{E_2^c}\left(\phi\right)$ is negative definite; and the \emph{conformal-biharmonic nullity} of $\phi$, which is the dimension of the kernel of $J_2^c$, is also finite.

\section{c-biharmonic stability of the identity map of Einstein manifolds}

We consider $M^m$ to be an Einstein manifold of nonnegative scalar curvature with $m\geq 4$  and $\phi$ to be the identity map $\Id:\left(M,g\right)\to\left(M,g\right)$. Thus,
$$
\Ric=\lambda g, \qquad \Scal=m\lambda, \qquad \lambda\geq0,
$$ 
and
$$
\nabla^{\Id}=\nabla, \qquad \Delta^{\Id}X=-\tr \left(\nabla^2 X\right)\left(\cdot,\cdot\right),
$$
where $X$ is a vector field on $M$, i.e., $X\in C(TM)$.

The identity map $\Id$ is harmonic, i.e., $\tau(\Id)=0$, and we can check that the Jacobi operators of the energy, bienergy and c-bienergy functionals, respectively, are given by 
$$
J(X)=\Delta_H X-2\lambda X, \qquad J_2(X)=J^2(X),
$$
and 
$$
J_2^c(X)=J^2(X)+\frac{2}{3}(m-3)\lambda J(X).
$$
Here, $\Delta_H$ is the Hodge Laplacian acting on vector fields and in the case of Einstein manifolds it is defined by
$$
\Delta_H X=-\tr \left(\nabla^2 X\right)\left(\cdot,\cdot\right)+\lambda X.
$$
We also recall that the set of all vector fields on $M$ can be written as the $L^2$-orthogonal sum
\begin{equation}\label{CTM-decomposition}
C(TM)=\left\{\nabla\alpha\ |\ \alpha\in C^\infty(M)\right\}\oplus\left\{X\in C(TM)\ |\ \Div X=0 \right\}.
\end{equation}
The case $m=3$ or the Ricci-flat case are trivial for our study as $J_2^c=J^2$ and thus, the identity map of a three-dimensional Einstein manifold is stable as a c-biharmonic map. From now on, we shall study the c-biharmonic stability of the identity map of a compact Einstein manifold of positive scalar curvature and with $m\geq 4$.

Assume that $\Delta_H X=\mu X$, with $\mu>0$. Thus, we have
$$
J(X)=(\mu-2\lambda)X, \qquad J_2(X)= (\mu-2\lambda)^2 X,
$$
and
$$
J_2^c(X)= (\mu-2\lambda)\left(\mu-\frac{2}{3}(6-m)\lambda\right)X.
$$
It follows that $J_2^c$ preserves the subspaces of the decomposition \eqref{CTM-decomposition} and we need to analyze two cases.
\begin{enumerate}
\item If $X=\nabla \alpha$, for \(\alpha\in C^\infty(M)\) with $\Delta \alpha=\mu \alpha$ and $\mu>0$, we deduce that 
$$
(\mu-2\lambda)\left(\mu-\frac{2}{3}(6-m)\lambda\right) < 0
$$
if and only if one of the following holds
\begin{itemize}
	\item[i)] $m=4$ and $\mu\in\left(4\lambda/3,2\lambda\right)$;
	\item[ii)] $m=5$ and $\mu\in\left(2\lambda/3,2\lambda\right)$;
	\item[iii)] $m\geq 6$ and $\mu\in (0,2\lambda)$.
\end{itemize}
Thus, if $m\geq 6$ the harmonic and c-biharmonic index coincide. For $m=4$ or $m=5$, the harmonic index could be greater than the c-biharmonic index, as there could exist $\mu\in\left(0,4\lambda/3\right]$ or $\mu\in\left(0,2\lambda/3\right]$, respectively. But, using the Lichnerowicz-Obata inequality (see \cite{Lichnerowicz1958, Obata1962}), i.e., 
$$
\mu\geq \frac{m}{m-1}\lambda,
$$ 
with equality if and only if $M$ is the round sphere of radius $\sqrt{(m-1)/\lambda}$, we conclude that this case occurs only when $m=4$. In this particular dimension, there exists $\mu=4\lambda/3$ which contributes to the $\Index_{E}\left(\Id\right)$ but not to the $\Index_{E_2^c}\left(\Id\right)$.
\item If $\Div X=0$, then $\mu\geq 2\lambda>0$ (see \cite{MR0375386} and also \cite{MR1252178}) and it follows that for any $m\geq 4$, we have
$$
(\mu-2\lambda)\left(\mu-\frac{2}{3}(6-m)\lambda\right) \geq 0.
$$
\end{enumerate}
We now conclude with the following result which is the main result of the article.

\begin{Theorem}\label{main-theorem}
Let $\left(M^m,g\right)$ be a compact Einstein manifold of positive scalar curvature with $m\geq 4$ and consider $\Id:\left(M,g\right)\to \left(M,g\right)$ to be the identity map. Then,
$$
\Index_{E_2^c}\left(\Id\right)=\Index_{E}\left(\Id\right) \qquad \text{and} \qquad
\Nullity_{E_2^c} \left(\Id\right)=\Nullity_{E_2} \left(\Id\right),
$$
with only one exception, namely up to a rescaling of the round metric, $\left(M,g\right)=\mathbb{S}^4$.
In this case,
$$
\Index_{E_2^c}\left(\Id\right)=0<\Index_{E}\left(\Id\right)=5 \qquad \text{and} \qquad \Nullity_{E_2^c}\left(\Id\right)=15>\Nullity_{E}\left(\Id\right)=10.
$$
\end{Theorem}

Using \cite[Proposition 2.13]{MR0375386}, we obtain the next Corollary.

\begin{Cor}
Let $\left(M^m,g\right)$ be a classical compact irreducible symmetric space with $m\geq 4$. Then,
$$
\Index_{E_2^c}\left(\Id\right)=0,
$$
with the following exceptions
\begin{itemize}
	\item[i)] $\mathbb{S}^m$, $m\geq 5$.
	\item[ii)] $Sp(p+q)/\left(Sp(p)\times Sp(q)\right)$, $p,q\geq 1$, $(p,q)\neq (1,1)$.
	\item[iii)] $SU(2p)/Sp(p)$, $p\geq 2$.
\end{itemize}
\end{Cor}

As a final contribution of this paper, we obtain an interesting result concerning
the identity map of $\mathbb{S}^m$ in dimensions at least five. First, we notice that, for $m$ at least four, we have

\begin{Prop}\label{prop-non-negativeE2c}
Let $\phi:\mathbb{S}^m\to N^n$ be a smooth map with $M$ compact and $m\geq 4$. Then, $E_2^c(\phi)\geq 0$, and the equality holds if and only if $\phi$ is a constant map.
\end{Prop}

Then, for $m$ at least five, adapting the techniques presented in \cite[Section 5]{MR703510} for harmonic maps, we show the following statement.

\begin{Theorem}\label{prop-e2c-small}
For any $\varepsilon>0$, there exists a map $\phi$ homotopic to the identity map $\Id:\mathbb{S}^m\to\mathbb{S}^m$, $m\geq 5$, with $E_2^c(\phi)<\varepsilon$.
\end{Theorem}

\begin{proof}
We identify $\mathbb{S}^m$ with $\left(\mathbb{S}^{m-1}\times(0,\pi),\sin^2r\cdot g_{\mathbb{S}^{m-1}}+dr^2\right)$ to which we add the north and south poles. Furthermore, for any $t>0$ we define the map
\begin{align*}
    \phi_t:\left(\mathbb{S}^{m-1}\times(0,\pi),\sin^2r\cdot g_{\mathbb{S}^{m-1}}+dr^2\right)
    \to\left(\mathbb{S}^{m-1}\times(0,\pi),\sin^2\alpha\cdot g_{\mathbb{S}^{m-1}}+d\alpha^2\right)
\end{align*}
by
\begin{align*}
    \phi_t(\theta,r)=\left(\theta,\alpha_t(r)\right),\qquad \forall \theta\in\mathbb{S}^{m-1}, \ r\in(0,\pi),
\end{align*}
where
\begin{align*}
    \alpha_t(r)=2\arctan\left(t\cdot\tan\left(\frac{r}{2}\right)\right).
\end{align*}
This map extends smoothly on the poles, i.e., $\alpha_t:[0,\pi]\to[0,\pi]$ is smooth and its derivatives of even order satisfy
\begin{equation*}
  \alpha_t(0)=0,\qquad\alpha_t(\pi)=\pi,\qquad\alpha_t^{(2k)}(0)=\alpha_t^{(2k)}(\pi)=0,
\end{equation*} 
where $k$ is a positive integer. Also, we note that $\phi_1=\Id$. 

A direct calculation shows that the first and the second derivatives of $\alpha_t$ are
$$
\alpha_t'(r)=\frac{\sin\left(\alpha_t(r)\right)}{\sin r}\qquad\text{and}\qquad \alpha_t''(r)=\frac{\sin\left(\alpha_t(r)\right)\left(\cos\left(\alpha_t(r)\right)-\cos r\right)}{\sin^2 r}.
$$
Next, we obtain
\begin{equation}\label{DifferentialNorm}
  \left|d\phi_t\right|^2=m\left(\frac{\sin\left(\alpha_t(r)\right)}{\sin r}\right)^2
\end{equation}
and from
$$
\tau(\phi_t)=\left(\alpha_t''(r)+(m-1)\cot r\cdot\alpha_t'(r)-(m-1)\frac{\sin\left(\alpha_t(r)\right)\cos\left(\alpha_t(r)\right)}{\sin^2 r}\right)\partial_\alpha,
$$
we deduce
\begin{align}
\left|\tau(\phi_t)\right|^2 &=\left(\alpha_t''(r)+(m-1)\cot r\cdot\alpha_t'(r)-(m-1)\frac{\sin\left(\alpha_t(r)\right)\cos\left(\alpha_t(r)\right)}{\sin^2 r}\right)^2 \nonumber\\
&= (m-2)^2\left(\frac{\sin\left(\alpha_t(r)\right)}{\sin r}\right)^2\left(\frac{\cos\left(\alpha_t(r)\right)-\cos r}{\sin r}\right)^2. \label{TensionFieldNorm}
\end{align}
Since
\begin{equation}\label{ConformalBienergyFunctional}
  E_2^c\left(\phi_t\right)=\frac{1}{2}\int_{\mathbb{S}^{m-1}\times[0,\pi]} \left(\left|\tau(\phi_t)\right|^2+\frac{2}{3}(m-1)(m-3)\left|d\phi_t\right|^2\right)\ v_g,
\end{equation}
replacing \eqref{DifferentialNorm} and \eqref{TensionFieldNorm} in \eqref{ConformalBienergyFunctional}, we have
\begin{align}\label{expr-E2c}
E_2^c(\phi_t)&=\frac{1}{2} \left(\int_{\mathbb S^{m-1}}v_{g_{\mathbb S^{m-1}}}\right) \nonumber\\
&\quad\times\int_0^\pi \sin^{m-1} r\cdot\left(\frac{\sin\left(\alpha_t(r)\right)}{\sin r}\right)^2\left((m-2)^2\left(\frac{\cos\left(\alpha_t(r)\right)-\cos r}{\sin r}\right)^2+\frac{2}{3}m(m-1)(m-3)\right)\ dr \nonumber\\
&=\frac{1}{2}\omega_{\mathbb S^{m-1}} \nonumber\\
&\quad\times\int_0^\pi\sin^{m-3}r\cdot\sin^2\left(\alpha_t(r)\right)\left((m-2)^2\left(\frac{\cos\left(\alpha_t(r)\right)-\cos r}{\sin r}\right)^2+\frac {2}{3}m(m-1)(m-3)\right)\ dr \nonumber\\
&=\frac{1}{2}\omega_{\mathbb S^{m-1}} \nonumber\\
&\quad\times \int_0^\pi\sin^2\left(\alpha_t(r)\right)\left((m-2)^2\sin^{m-5}r\cdot\left(\cos\left(\alpha_t(r)\right)-\cos r\right)^2+\frac{2}{3}m(m-1)(m-3)\sin^{m-3}r\right)\ dr. 
\end{align}
Using the upper bounds for $\sin$ and $\cos$, since $m\geq 5$, we get the following estimate for the c-bienergy functional
\begin{align}
E_2^c(\phi_t)&<\frac{1}{2}\omega_{\mathbb S^{m-1}}\int_0^\pi\sin^2\left(\alpha_t(r)\right)\left(4(m-2)^2+\frac{2}{3}m(m-1)(m-3)\right)\ dr \nonumber\\
&=\left(2(m-2)^2+\frac{1}{3}m(m-1)(m-3)\right)\omega_{\mathbb S^{m-1}}\int_0^\pi\sin^2\left(\alpha_t(r)\right)\ dr \nonumber\\
&= C \int_0^\pi\sin^2\left(\alpha_t(r)\right)\ dr, \label{EstimateConformalBienergy}
\end{align}
where 
$$
C=\left(2(m-2)^2+\frac{1}{3}m(m-1)(m-3)\right)\omega_{\mathbb S^{m-1}}>0
$$
and $\omega_{\mathbb S^{m-1}}$ denotes the volume of the Euclidean unit sphere $\mathbb{S}^{m-1}$.
 
For the last part of the proof, we follow the ideas from \cite{MR703510}. We consider an arbitrary positive real number $\varepsilon$ and define the positive constants 
\begin{equation*}
  \eta_\varepsilon=\frac{\varepsilon}{C}\qquad\text{and}\qquad\rho_\varepsilon=\pi-\frac{\eta_\varepsilon}{2}.
\end{equation*}
The function $\tan$ is increasing thus, denoting by $K_\varepsilon=\tan\left(\rho_\varepsilon/2\right)>0$, we get
\begin{equation*}
  \tan\frac{r}{2}\in\left(0,K_\varepsilon\right),\qquad\forall r\in\left(0, \rho_\varepsilon\right).
\end{equation*}
Since $\lim_{x\to0}\sin^2x=0$, there exists $\delta_\varepsilon>0$ such that 
\begin{equation*}
  \sin^2x\in\left(0, \frac{\eta_\varepsilon}{2\rho_\varepsilon}\right),\qquad\forall x\in\left(0,\delta_\varepsilon\right).
\end{equation*}
We can assume that $\delta_\varepsilon<\pi/2$. Further, from $\lim_{t\to 0}\left(2\arctan\left(tK_\varepsilon\right)\right)=0$, it follows that there exists $\delta'_\varepsilon>0$ such that
\begin{equation*}
  2\arctan\left(tK_\varepsilon\right)\in\left(0,\delta_\varepsilon\right),\qquad\forall t\in\left(0,\delta'_\varepsilon\right).
\end{equation*}
Thus,
\begin{equation*}
  \sin^2\left(2\arctan\left(tK_\varepsilon\right)\right)\in\left(0,\frac{\eta_\varepsilon}{2\rho_\varepsilon}\right),\qquad\forall t\in\left(0,\delta'_\varepsilon\right),
\end{equation*}
and we achieve 
\begin{equation}\label{EstimateSin}
  0<\sin^2\left(\alpha_t(r)\right)=\sin^2\left(2\arctan\left(t\cdot\tan\left(\frac{r}{2}\right)\right)\right)<\sin^2\left(2\arctan\left(tK_\varepsilon\right)\right)<\frac{\eta_\varepsilon}{2\rho_\varepsilon},
\end{equation}
for any $t\in\left(0,\delta'_\varepsilon\right)$ and any $r\in\left(0,\rho_\varepsilon\right)$.

Replacing \eqref{EstimateSin} in \eqref{EstimateConformalBienergy}, we obtain
\begin{align*}
   E_2^c\left(\phi_t\right)&<C\left(\int_{0}^{\rho_\varepsilon} \sin^2\left(\alpha_t(r)\right)\ dr+\int_{\rho_\varepsilon}^{\pi} \sin^2\left(\alpha_t(r)\right)\ dr\right)\\
   &<C\left(\int_{0}^{\rho_\varepsilon} \frac{\eta_\varepsilon}{2\rho_\varepsilon}\ dr+\int_{\rho_\varepsilon}^{\pi} \ dr\right)\\
   &=\varepsilon.
\end{align*}
\end{proof}

Using the notations from the proof of Proposition \ref{prop-e2c-small}, we get the next result.
\begin{Cor} 
We consider the function $h_m^c:(0,\infty)\to(0,\infty)$ given by $h_m^c(t):=E_2^c\left(\phi_t\right)$. Then, we have
\begin{itemize}
    \item[i)] if $m=4$, the function $h^c_m$ is constant;
    \item[ii)] if $m\geq 5$, the point $t=1$ is a local maximum of $h_c^m$ and $\lim_{t\to 0} h^c_m(t)=0$.
\end{itemize}
\end{Cor}

\begin{proof}
	Let us consider the case $m=4$. Since the map $\phi_t$ is a conformal diffeomorphism of $\mathbb{S}^4$, it follows from \cite[Theorem 2.14]{BNO} that
	$$
	E_2^c(\phi_t) = E_2^c(\Id),
	$$
	where $\Id = \phi_1$ denotes the identity map on $\mathbb{S}^4$. Consequently, the function $h_4^c$ is constant.
	
	We now provide a direct proof of this fact. Starting from \eqref{expr-E2c}, which holds for any $m\geq 2$ and, in particular for $m=4$, we perform the change of variables
	$$
	x = \tan\frac{r}{2} \in (0,\infty).
	$$
	A straightforward computation yields
	$$
	h^c_4(t)= \frac{1}{2}\omega_{\mathbb S^{3}}\int_0^\infty 
	\frac{64x^3t^2\left(x^2\left(1-t^2\right)^2+2\left(1+t^2x^2\right)^2\right)}
	{\left(1+t^2x^2\right)^4\left(1+x^2\right)^2} \, dx.
	$$
	Next, a direct computation shows that
	$$
	\frac{\partial}{\partial t}\left\{
	\frac{64x^3t^2\left(x^2\left(1-t^2\right)^2+2\left(1+t^2x^2\right)^2\right)}
	{\left(1+t^2x^2\right)^4\left(1+x^2\right)^2}
	\right\}=\frac{\partial}{\partial x}\left\{\frac{64tx^4\left(1+t^4x^2\right)}
	{\left(1+t^2x^2\right)^4\left(1+x^2\right)}\right\}.
	$$
	
	Therefore, differentiating under the integral, we obtain
	$$
	\frac{d}{dt}\left\{h_4^{c}(t)\right\}=0, \qquad \forall t>0.
	$$
	Hence, $h_4^c$ is constant and, more precisely, since $\omega_{\mathbb S^{3}}=2\pi^2$ and $\omega_{\mathbb S^{4}}=8\pi^2/3$, we get  
	\begin{align*}
	h_4^c(t)&=h_4^c(1)\\
	&= \frac{1}{2}\omega_{\mathbb S^{3}}\cdot \frac{32}{3}\\
	&=4\omega_{\mathbb{S}^4}, \qquad \forall t>0.
	\end{align*}
	Further, let us consider the case $m\geq 5$.  If we denote by $W$ the variational vector field associated with $\left\{\phi_t\right\}_{t>0}$, i.e., 
	$$
		W(p)= \left.\frac{d}{dt}\right|_{t=1}\left\{\phi_{t}(p)\right\}=(\sin r)\partial_r\in T_{\phi(p)}\mathbb{S}^m, \qquad\forall p\equiv(\theta,r)\in\mathbb{S}^m,
	$$ 
	we can see that $W=\nabla \gamma=\gamma'(r)\partial_r$, where $\gamma=\gamma(r)=-\cos r$. By straightforward computations, we obtain
	$$
	\Delta \gamma=-\left(\gamma''+(m-1)\gamma'\cot r\right)=m\gamma.
	$$
	Since
	$$
	\left.\frac{d^2}{dt^2}\right|_{t=1}\left\{h_m^{c}(t)\right\}=\int_{\mathbb{S}^m}\langle J_2^c(W),W\rangle\ v_g,
	$$
	from \cite[Corollary 4.3]{BNO}, where the c-biharmonic index of the identity map of $\mathbb{S}^m$ was computed, we obtain 
	$$
	\left.\frac{d^2}{dt^2}\right|_{t=1}\left\{h_m^{c}(t)\right\}<0.
	$$
	Thus, $t=1$ is a local maximal point of $h_m^c$.
	
	A direct proof of this fact could also be given; however, the computations are rather lengthy, so we omit them.
	
	Finally, from the proof of Proposition \ref{prop-e2c-small} it is not difficult to notice that
	$$
	\lim_{t\to 0} h^c_m(t)=0.
	$$
\end{proof}

Now, we can extend the result from Proposition \ref{prop-e2c-small} to the homotopy classes of harmonic maps from a sphere $\mathbb{S}^m$, $m\geq 5$, to an arbitrary manifold $N^n$.

\begin{Theorem}
Let $\varphi:\mathbb{S}^m\to N^n$ be a harmonic map, $m\geq 5$.
Then, the infimum of $E_2^c$ in the homotopy class of $\varphi$ is zero.
\end{Theorem}

\begin{proof}
Using the notations from the proof of Proposition \ref{prop-e2c-small}, we note that $\varphi \circ \phi_t$ is homotopic to $\varphi \circ \Id = \varphi$. Then, since
\begin{align*}
\tau(\varphi\circ\phi_t)&=d\varphi\left(\tau(\phi_t)\right)+\tr\nabla d\varphi\left(d\phi_t(\cdot), d\phi_t(\cdot)\right)\\
&=d\varphi\left(\tau(\phi_t)\right),
\end{align*}
we deduce that
\begin{equation*}
\left|\tau(\varphi\circ\phi_t)\right|^2\leq\left|d\varphi\right|^2\left|\tau(\phi_t)\right|^2\leq\max_{p\in\mathbb S^m}\left|d\varphi\right|^2\cdot\left|\tau(\phi_t)\right|^2.
\end{equation*}
Taking into account
\begin{equation*}
\left|d(\varphi\circ\phi_t)\right|^2\leq\max_{p\in\mathbb S^m}\left|d\varphi\right|^2\cdot\left|d\phi_t\right|^2,
\end{equation*}
we get that
\begin{equation*}
E_2^c\left(\varphi\circ\phi_t\right)\leq\max_{p\in\mathbb S^m} \left|d\varphi\right|^2\cdot E_2^c\left(\phi_t\right).
\end{equation*}
Now, applying Proposition \ref{prop-e2c-small}, the conclusion follows.
\end{proof}

\begin{Bem}
If the homotopy class of maps associated with \(\phi\) is not trivial,
then the \(c\)-biharmonic maps homotopic to \(\phi\) (if there are others than \(\phi\) itself) do not minimize \(E_2^c\); they are critical points of \(E_2^c\) but not minimizers.
\end{Bem}

\bibliographystyle{plain}
\bibliography{mybib}
\end{document}